\documentclass[a4paper,12pt,reqno]{amsart}
\usepackage{amsthm, mathrsfs,amssymb,amsmath,}
\usepackage{enumerate}
\usepackage{mathtools}
\makeatletter
\@namedef{subjclassname@2010}{%
	\textup{2010} Mathematics Subject Classification}
\makeatother

\usepackage{hyperref}
\allowdisplaybreaks
\numberwithin{equation}{section}

\usepackage{graphicx}
\usepackage{enumitem}

\newtheorem{theorem}{Theorem}[section]
\newtheorem{lemma}[theorem]{Lemma}
\usepackage{xcolor}

\theoremstyle{definition}

\theoremstyle{remark}
\newtheorem{remark}[theorem]{Remark}

\numberwithin{equation}{section}

\begin{document}
	
	\title[Some Topological Properties and bi-Lipschitz equivalence of GDA]{Some Topological Properties and bi-Lipschitz equivalence of Graph-Directed Attractors}

		\author{Subhash Chandra}
	
	\address{Faculty of Mathematics and Computer Science, Transilvania University of Braşov, Iuliu Maniu Street, nr. 50, 500091, Braşov, Romania}

	\email{sahusubhash77@gmail.com}

	\subjclass{Primary 28A80; Secondary 05C20}
	\keywords{Graph directed IFS, attractor,  Strong separation condition, Open set condition, bi-Lipschitz map}
	
	\begin{abstract}In this paper, we discuss some topological properties of the graph-directed iterated function system (GDIFS) of injective contractions. Further, we establish the existence of a Lipschitz bijection between two different bi-Lipschitz graph-directed attractors. In general, two different bi-Lipschitz graph-directed attractors defined on the same graph need not be bi-Lipschitz equivalent. However, under suitable conditions, we prove that these graph-directed attractors are bi-Lipschitz equivalent.
	\end{abstract}
	\maketitle

	
	
	\section{Introduction}
    Many classes of fractals display the property of self-similarity, which means that they are composed of infinitely scaled-down copies of the entire set. Hutchinson \cite{HT} established a rigorous framework for self-similar sets and introduced their mathematical construction through Iterated Function Systems (IFSs). This approach was later developed and widely popularized by Barnsley and his collaborators \cite{MF1, MF2,MF3}, leading to extensive studies on the theory and applications of IFSs. An IFS on a closed subset $X$ of an Euclidean space is essentially a finite family of contractions  
$(f_i: X \to X )_{i \in I}.$ Its attractor is the unique non-empty compact subset A of X such that   $A=\cup_{i \in I} f_i(A)$.

Let us present a generalization of the concept of attractor of an IFS that was introduced by Mauldin and Williams \cite{MW}\ , under the label of graph directed attractor, via the concept of graph-directed iterated function system (for short GDIFS).

Let $G(V,E)$ be a directed graph such that for each $i\in V$ there is $e\in E
$ starting from $i$ (here $V=\{1,...,n\}$ is the set of vertices and $E$ is
the finite set of directed edges) and $(S_{e})_{e\in E}$ a family of
contractions from $\mathbb{R}^{d}$ to $\mathbb{R}^{d}$. The pair $\mathcal{S}=(G(V,E),(S_{e})_{e\in E})$ is called a graph directed
iterated function system in $\mathbb{R}^{d}$ and for such a system there
exists a unique $(A_{1},...,A_{n})\in (P_{cp}(\mathbb{R}^{d}))^{n}$ -called
the attractor of $\mathcal{S}$- such that \begin{equation}\label{ATR}A_{i}=\overset{n}{\underset{k=1}{%
\cup }}\underset{e\in E_{i,k}}{\cup }S_{e}(A_{k})\end{equation} for each $i\in \{1,...,n\}
$, where $P_{cp}(\mathbb{R}^{d})$ designates the family of nonempty compact
subsets of $\mathbb{R}^{d}$ and $E_{i,k}$ stands for the set of edges from $i
$ to $k$. The attractors of GDIFSs are called graph directed attractors.

The set $\overset{n}{\underset{i=1}{\cup }}A_{i}:=
\mathbf{A}$ is called the graph direceted set of $\mathcal{S}$.

For $i,j\in V$ and $p\in \mathbb{N}$, by $E_{i,j}^{p}$ we denote the set of
all direct paths $\mathbf{e}=(e_{1},...,e_{p})$ from $i$ to $j$, i.e. $e_{1}$
is a direct edge starting from $i$, $e_{p}$ is a direct edge finishing in $j$
and $e_{k}$ is a direct edge from $i_{k-1}$ to $i_{k}$ for all $k\in
\{2,...,p\}$, with $i_{1}=i$ and $i_{p}=j$.

For $\mathbf{e}=(e_{1},...,e_{p})\in E_{i,j}^{p}$, by $S_{\mathbf{e}}$ we
mean $S_{e_{1}}\circ ...\circ S_{e_{p}}$.

For $i,j\in V$, by $E_{i,j}$ we denote $\underset{p\in \mathbb{N}}{\cup }%
E_{i,j}^{p}$.

Note that%
\begin{equation*}
A_{i}=\underset{k=1}{\overset{n}{\cup }}\underset{\mathbf{e}\in E_{i,k}^{p}}{%
\cup }S_{\mathbf{e}}(A_{k})\text{,}
\end{equation*}%
for all $i\in V$ and $p\in \mathbb{N}$.
For $i\in V$, we denote by $E_i^{*}$ the set of all infinite admissible paths
starting from the vertex $i$, that is,
\[
E_i^{*}:=
\left\{(e_1,e_2,\ldots):\ e_m\in E_{i_m,i_{m+1}}
\text{ for all }m\in\mathbb{N},\ \text{with } i_1=i\right\}.
\]

\begin{remark}\label{RMK1} \textit{If} $\mathcal{S}=(G(V,E),(S_{e})_{e\in E})$ 
\textit{is a GDIFS in} $\mathbb{R}^{d}$ \textit{such that all }$S_{e}$%
\textit{\ are one-to-one} \textit{and }$(A_{1},...,A_{n})\in (P_{cp}(\mathbb{%
R}^{d}))^{n}$ \textit{its attractor, then}%
\[
S_{e}({A_{i}}^\circ)=({S_{e}(A_{i})})^\circ\text{,}
\]%
\textit{for all }$e\in E$\textit{\ and all }$i\in \{1,...,n\}=V$\textit{.}

This is clear since the function $f_{i,e}:A_{i}\rightarrow S_{e}(A_{i})$,
given by $f_{i,e}(x)=S_{e}(x)$ for all $x\in A_{i}$, is a homomorphism (as $%
A_{i}\in P_{cp}(\mathbb{R}^{d})$ and $f_{i,e}$ is a continuous bijection)
for all $e\in E$ and all $i\in \{1,...,n\}=V$\textit{.}\end{remark}

A GDIFS $\mathcal{S}=(G(V,E),(S_{e})_{e\in E})$ is called strongly connected
if the direct graph $G(V,E)$ is strongly connected, i.e. for every vertices $%
i,j\in V$ there exists a direct path from $i$ to $j$.

\textit{Throughout this article, we consider strongly connected GDIFS}.

A GDIFS $\mathcal{S}=(G(V,E),(S_{e})_{e\in E})$ satisfies the strong
separation condition (for short SSC) if, for each $i\in V$, the sets $%
S_{e}(A_{k})$, where $k\in V$ and $e\in E_{i,k}$, are disjoint.

A GDIFS $\mathcal{S}=(G(V,E),(S_{e})_{e\in E})$ satisfies the open set
condition (for short OSC) if there exists $(O_{1},...,O_{n})$, where $%
O_{k}$ is a nonempty open bounded subset of $\mathbb{R}^{d}$ for all $k\in \{1,...,n\}$,
such that, for each $i\in \{1,...,n\}$, we have $\overset{n}{\underset{k=1}{%
\cup }}\underset{e\in E_{i,k}}{\cup }S_{e}(O_{k})\subseteq O_{i}$ and the
union is disjoint.

Note that a GDIFS that satisfies SSC also satisfies OSC.
For details concerning GDIFSs, see \cite{GAE, AF, MW}.

The study of topological structure of self-similar and self-affine sets, including connectedness, local connectedness, and disk-likeness, is an important topic in fractal geometry. Hata initially provided a criterion for the connectivity of self-similar sets in \cite{HATA}. As a result, a lot of research has been done on these topological characteristics (\cite{CB, QR, QR2, IK,JHB,SM}). In \cite{GTD, GT2}, Deng and his collaborators have presented some interesting results on self-affine tiles. The topological properties of the class of self-affine tiles in $\mathbb{R}^3$ were covered in \cite{GTD}, while the case of the higher dimension was covered in \cite{GT2}. We also refer to some recent developments in nearby areas \cite{SD, MK, MNV}.\\

In the study of fractal geometry, the Lipschitz equivalence of different attractors is a crucial topic, and the bi-Lipschitz equivalence of self-similar sets has received a lot of attention. A significant step in the algebraic study of Lipschitz equivalence was taken by Falconer and Marsh \cite{FM}, who showed that for two dust-like self-similar sets to be Lipschitz equivalent, their scaling ratios must satisfy certain multiplicative relations, thereby uncovering strong algebraic links within the underlying IFSs. This direction was advanced in \cite{Rao}, where the first author, together with Rao and Wang, proposed the matchable condition as a necessary criterion for Lipschitz equivalence, leading to a complete classification in two notable situations: when the contraction vectors are of full rank, or when the system contains precisely two maps. Subsequently, Rao and Zhang \cite{Rao2} connected the classification problem with a higher-dimensional Frobenius problem, which provided conditions for equivalence in cases where the scaling ratios are coplanar in a specific sense. More recently, Xi and Xiong \cite{Xi} established a verifiable criterion for systems with commensurable ratios.
In \cite{HJR} Ruan et al. investigated the Lipschitz equivalence of self-similar sets in $\mathbb{R}$ that exhibit touching structures with an arbitrary number of branches. 
Further related developments can be found in \cite{JJ, ML, RAO, LF, YX} and the references therein.
The majority of the aforementioned works study the Lipschitz equivalence for self-similar sets. 
It is well known that the attractors associated with GDIFS exhibit behaviors that differ from those of the classical IFS \cite{BF, FHZ}. In particular, the study of GDIFS associated with bi-Lipschitz mappings reveals several intricate phenomena, and the corresponding results may be regarded as natural generalizations of those concerning self-similar sets. This is one of the primary reasons to work in this research direction.

Now, in order to set the stage for this paper we provide some fundamental concepts and results.
\par
 Given $\delta >0$ and a non-empty subset $E$ of $\mathbb{R}^{n}$, an at most
countable family $(F_{i})_{i\in I}$ of subsets of $\mathbb{R}^{n}$ is called
a $\delta $-cover of $E$ if $E\subseteq \underset{i\in I}{\cup }F_{i}$ and $%
diam(F_{i})<\delta $ for all $i\in I$, where for a non-empty subset $F$ of $%
\mathbb{R}^{n}$, by $diam(F)$ we designate the diameter of $F$.

If we consider $s>0$, then $\inf \{\underset{i\in I}{\sum }%
(diam(F_{i}))^{s}\mid E\subseteq \underset{i\in I}{\cup }F_{i} \text{ and }  
diam(F_{i})<\delta  \text{ for all } i\in I\}$  is denoted by $H_{\delta }^{s}(E)$ and $\underset{%
\delta \rightarrow 0}{\lim }H_{\delta }^{s}(E)$ represents the $s$%
-dimensional Hausdorff measure of $E$ which is denoted by $H^{s}(E)$. In
addition $\sup \{s>0\mid H^{s}(E)=\infty \}$ (which is the same with $\inf
\{s>0\mid H^{s}(E)=0\}$) defines the Hausdorff dimension of $E$ which is
denoted by $\dim _{H}(E)$.

Moreover, $\underset{\delta \rightarrow 0}{\underline{\lim }}\frac{\ln
N_{\delta }(E)}{-\ln \delta }$ defines the lower box dimension of $E$, which
is denoted by $\underline{\dim }_{B}(E)$, and $\underset{\delta \rightarrow 0%
}{\overline{\lim }}\frac{\ln N_{\delta }(E)}{-\ln \delta }$ defines the
upper box dimension of $E$, which is denoted by $\overline{\dim }_{B}(E)$,
and if they coincide, we define the box dimension of $E$ as being their common
value, i.e. $\dim _{B}(E)=\underset{\delta \rightarrow 0}{\lim }\frac{\ln
N_{\delta }(E)}{-\ln \delta }$, where $N_{\delta}(E)$ be the smallest number of sets of diameter at most $\delta$  which cover $E.$ 

Details concerning the above concepts can be found in \cite{Fal}.\\
Given a metric space $(X,d)$ and a function $f:X \to X, ~ \sup \frac{d(f(x),f(y))}{d(x,y)}=\operatorname{Lip}^+(f)$ and $\inf \frac{d(f(x),f(y))}{d(x,y)}=\operatorname{Lip}^-(f)$ designate the upper Lipschitz constant and, respectively, the lower Lipschitz constant. If $0<\operatorname{Lip}^-(f)\le \operatorname{Lip}^+(f)<\infty$, then $f$ is called bi-Lipschitz.

A GDIFS $S=(G(V,E),(S_{e})_{e\in E})$ satisfies the bounded distorsion
property (BDP for short) if there exists $K>0$ such that $\operatorname{Lip}^{+}(S_{\mathbf{%
e}})\leq K\operatorname{Lip}^{-}(S_{\mathbf{e}})$ for all $i,j\in V$ and all $e\in E_{i,j}$.
\pagebreak

\begin{lemma} \label{FLemma} \hfill
\begin{enumerate}
    \item[a)](see \cite{Fal}, page 32) \textit{If} $A\subseteq B\subseteq \mathbb{R}^{d}$\textit{, then}%
\[
\dim _{H}(A)\leq \dim _{H}(B)\text{.}
\]

    \item[b)](see \cite{Fal}, Corollary 2.4)  If $f: A \to \mathbb{R}^d$ where $A \subseteq \mathbb{R}^d$, is bi-Lipschitz, then 
 	$$ \dim_{H}(A)=\dim_{H}(f(A)).$$
\end{enumerate}
\end{lemma}

\begin{lemma} \label{equaldiem}
    If $\mathcal{S}=(G(V,E),(S_{e})_{e\in E})$ is a GDIFS in $\mathbb{R}^{d}$ and $(A_{1},...,A_{n})\in (P_{cp}(
\mathbb{R}^{d}))^{n}$ its attractor, then
$$\dim _{H}(A_{j})=\dim _{H}(A_{i}),$$
for all $i,j\in V$.
\end{lemma}
\begin{proof}
     Let us consider $i,j\in V$ arbitrarily chosen, but fixed.

As $\mathcal{S}$ is strongly connected, there exist a direct path $\mathbf{e}
$ from $i$ to $j$ and a direct path $\mathbf{e}^{^{\prime }}$ from $j$ to $i$
and, in view of $(1.1)$, we infer that 
\begin{equation}
S_{\mathbf{e}}(A_{j})\subseteq A_{i}  \tag{1}
\end{equation}%
and%
\begin{equation}
S_{\mathbf{e}^{^{\prime }}}(A_{i})\subseteq A_{j}\text{.}  \tag{2}
\end{equation}

Then we have%
\[
\dim _{H}(A_{j})\overset{S_{\mathbf{e}}\text{ is bi-Lipschitz \& Lemma 1.1,
b)}}{=}\dim _{H}(S_{\mathbf{e}}(A_{j}))\overset{\text{(1) and Lemma 1.1, a)}}%
{\leq }
\]%
\[
\leq \dim _{H}(A_{i})\overset{S_{\mathbf{e}^{^{\prime }}}\text{ is
bi-Lipschitz \& Lemma 1.1, b)}}{=}\dim _{H}(S_{\mathbf{e}^{^{\prime
}}}(A_{i}))\overset{\text{(2) and Lemma 1.1, a) }}{\leq }\dim _{H}(A_{j})%
\text{,}
\]%
so we conclude that%
\[
\dim _{H}(A_{i})=\dim _{H}(A_{j})\text{. }
\]

\bigskip 

\end{proof}
\begin{lemma}[see \cite{GT}, Separation Theorem, page 34] \label{GTL}
    If $V_1 \subset \mathbb{R}^2$ is compact, $V_2 \subset \mathbb{R}^2$ is a closed set with $V_1 \cap V_2$ totally disconnected and $v_1, v_2$ are points of $V_1 \backslash (V_1 \cap V_2)$ and $V_2\backslash (V_1 \cap V_2)$, respectively, and $\varepsilon$ is any positive number, then there exists a Jordan curve $\gamma$ which separates $v_1$ and $v_2$ and having the property that $\gamma \cap (V_1 \cup V_2) \subseteq V_1 \cap V_2$, and every point of $\gamma$ is at a distance less than $\varepsilon$ from some point of $V_1.$
\end{lemma}

A set $X \subset \mathbb{R}^2$ is said to have \textbf{a hole} if there exists a Jordan curve $J \subset X$ such that $J^0 \cap X=\emptyset.$

	\section{Main results}

\textbf{Proposition 2.1.} \textit{If} $\mathcal{S}=(G(V,E),(S_{e})_{e\in E})$
\textit{is a GDIFS in} $\mathbb{R}^{d}$ \textit{satisfying OSC such that }$%
S_{e}$\textit{\ is one-to-one for all }$e\in E$\textit{\ and }$%
(A_{1},...,A_{n})\in (P_{cp}(\mathbb{R}^{d}))^{n}$\textit{\ its attractor,
then}%
\begin{equation*}
S_{e}(A_{i}^{\circ })\cap S_{e^{^{\prime }}}(A_{j}^{\circ })=\emptyset \text{%
,}
\end{equation*}%
\textit{for all }$i,j,k\in \{1,...,n\}$\textit{, }$e\in E_{k,i}$\textit{\
and }$e^{^{\prime }}\in E_{k,j}$\textit{, with }$e\neq e^{^{\prime }}$%
\textit{.}

\textit{Proof}. As $\mathcal{S}$ satisfies OSC, there exists $%
(O_{1},...,O_{n})$, where $O_{k}$ is a nonempty subset of $\mathbb{R}^{d}$
for all $k\in \{1,...,n\}$, such that 
\begin{equation*}
\underset{k=1}{\overset{n}{\cup }}\underset{e\in E_{i,k}}{\cup }%
S_{e}(O_{k})\subseteq O_{i}\text{,}
\end{equation*}%
the union being disjoint, for each $i\in \{1,...,n\}$.

Adopting the notation $O_{i}\cap A_{i}^{\circ }{:=}U_{i}$%
, we have%
\begin{equation}
\overline{U_{k}}=A_{k}\text{,}  \tag{1}
\end{equation}%
for all $k\in \{1,...,n\}$.

Indeed, the inclusion $\overline{U_{k}}\subseteq A_{k}$ yield from the
definition of $U_{k}$ and the inclusion $A_{k}\subseteq \overline{U_{k}}$ is
stated in Lemma 1.3.6 from [5].

Let us suppose, by reductio ad absurdum, that there exist $i,j,k\in
\{1,...,n\}$, $e\in E_{k,i}$ and $e^{^{\prime }}\in E_{k,j}$, with $e\neq
e^{^{\prime }}$, such that 
\begin{equation}
S_{e}(A_{i}^{\circ })\cap S_{e^{^{\prime }}}(A_{j}^{\circ })\neq \emptyset 
\text{.}  \tag{2}
\end{equation}

Then, via $(2)$, there exist 
\begin{equation*}
x\in S_{e}(A_{i}^{\circ })\cap S_{e^{^{\prime }}}(A_{j}^{\circ })\overset{%
\text{Remark 1.1}}{=}(S_{e}(A_{i})\cap S_{e^{^{\prime }}}(A_{j}))^{\circ }
\end{equation*}%
and $\varepsilon >0$ such that 
\begin{equation*}
B(x,\varepsilon )\subseteq S_{e}(A_{i}^{\circ })\cap S_{e^{^{\prime
}}}(A_{j}^{\circ })\overset{\text{(1)}}{\subseteq }S_{e}(\overline{U_{i}}%
)\cap S_{e^{^{\prime }}}(\overline{U_{j}})\subseteq \overline{S_{e}(U_{i})}%
\cap \overline{S_{e^{^{\prime }}}(U_{j})}\text{,}
\end{equation*}%
so we have 
\begin{equation}
B[x,\varepsilon ]=\overline{B(x,\varepsilon )}\subseteq \overline{%
S_{e}(U_{i})}\cap \overline{S_{e^{^{\prime }}}(U_{j})}\text{.}  \tag{3}
\end{equation}

As, in view of $(3)$, the open sets $S_{e}(U_{i})\cap B(x,\varepsilon )$ and 
$S_{e^{^{\prime }}}(U_{j})\cap B(x,\varepsilon )$ have the property that 
\begin{equation*}
\overline{S_{e}(U_{i})\cap B(x,\varepsilon )}=B[x,\varepsilon ]
\end{equation*}%
and 
\begin{equation*}
\overline{S_{e^{^{\prime }}}(U_{j})\cap B(x,\varepsilon )}=B[x,\varepsilon ]%
\text{,}
\end{equation*}%
we infer that 
\begin{equation}
\overline{S_{e}(U_{i})\cap S_{e^{^{\prime }}}(U_{j})\cap B(x,\varepsilon )}%
=B[x,\varepsilon ]\text{.}  \tag{4}
\end{equation}

Hence, since 
\begin{equation*}
S_{e}(U_{i})\cap S_{e^{^{\prime }}}(U_{j})\subseteq S_{e}(O_{i})\cap
S_{e^{^{\prime }}}(O_{j})=\emptyset \text{,}
\end{equation*}%
taking into account $(4)$, we get the contradiction $\emptyset
=B[x,\varepsilon ]$ which completes the proof. $\square $
\begin{theorem}\label{Main}
    \textit{Let} $\mathcal{S}=(G(V,E),(S_{e})_{e\in E})$
\textit{is a GDIFS in} $\mathbb{R}^{2}$ \textit{satisfying OSC such that }$%
S_{e}$\textit{\ is one-to-one for all }$e\in E$\textit{\ and }$%
(A_{1},...,A_{n})\in (P_{cp}(\mathbb{R}^{2}))^{n}$\textit{\ its attractor.} If $A_i$ is connected for all $i\in V$, then the following statements hold:
    \begin{itemize}
        \item[(i)] If $A_{i}^0\ne \emptyset$ for some $i\in V$ then at least one $A_i$ has no hole.
        \item[(ii)] The boundary $\partial A_i$ of $A_i$ is connected.
    \end{itemize}
\end{theorem}
\begin{proof}
(i)
      Let $J_k \subset A_k^0$ be a simple closed curve for each $k\in V$. Let the ineterior $\mathcal{O}_k$ of $J_k$ is not contained in $A_k^0$, we have some $x_k \in \mathcal{O}_k \backslash A_{k}.$ Given $z \in A_i^0$, there exists an $\epsilon>0$ such that the ball $B[z,\epsilon]=\{y \in \mathbb{R}^2:\|y-z\| \le \epsilon\}$ is contained in $A_{i}$. 
    For $\mathbf{e}=(e_1,e_2,...,e_m) \in E_{i,j}^m$, we define  $S_{\mathbf{e}}:=S_{e_1} \circ S_{e_2} \circ \cdots \circ S_{e_m}.$
    One can see that $A_i=\bigcup_{k=1}^{n}\bigcup_{\mathbf{e}\in E_{i,k}^m}S_e(A_k)$ for each $m\in \mathbb{N}.$ Set $r:=\max \{r_{e}: e \in E \}$, where $r_e \in (0,1)$ is the contraction ratio of the map $S_e$ for all $e \in E$. Clearly $r\in (0,1)$. Then, there exists $N>0$ such that $r^N \text{diam} (\mathbf{A})< \frac{\epsilon}{2}.$ This implies $r^N \text{diam}(A_i) \le r^N \text{diam} (\mathbf{A}) < \frac{\epsilon}{2}$ for all $i\in V.$ Consequently, $\text{diam} (S_{\mathbf{e}}(A_k)) < \frac{\epsilon}{2}$ for all $\mathbf{e} \in E_{i,k}^N.$ Since $z \in A_{i}^{0}\subset A_i=\bigcup_{k=1}^{n}\bigcup _{\mathbf{e}\in E_{i,k}^N}S_{\mathbf{e}}(A_k)$, there exist $k_1\in \{1,2,\dots,n\}$ and $ \mathbf{e_1} \in E_{i,k_1}^N$ such that $z \in S_{\mathbf{e_1}}(A_{k_1}).$ We get $S_{\mathbf{e_1}}(J_{k_1}) \subset S_{\mathbf{e_1}}(A_{k_1})$ and $$S_{\mathbf{e_1}}(x_{k_1}) \in S_{\mathbf{e_1}} (\mathcal{O}_{k_1}) \backslash S_{\mathbf{e_1}}(A_{k_1})  \subseteq B[z, \epsilon].$$ Thus, we have $S_{\mathbf{e_1}}(x_{k_1}) \notin S_{\mathbf{e_1}}(A_{k_1})$ and $S_{\mathbf{e_1}}(x_{k_1})\in A_i$. \\ Again,  since $A_i=\bigcup\limits_{k=1}^{n} \bigcup\limits_{\mathbf{e}\in E_{i,k}^N}S_{\mathbf{e}}(A_k)$, there is an $k_2 \in V$ and $\mathbf{e_2} \in E_{i,k_2}^N $ such that $ \mathbf{e_2}\ne \mathbf{e_1}$ and $S_{\mathbf{e_1}}(x_{k_1}) \in S_{\mathbf{e_2}}(A_{k_2}).$
   Since $J_{k_1} \subset A_{k_1}^{\circ}$, we have 
$S_{\mathbf{e_1}}(J_{k_1}) \subset S_{\mathbf{e_1}}(A_{k_1}^{\circ})$.  
By Proposition~2.1 the interiors of different cylinder images are disjoint, that is,
\[
S_{\mathbf{e_1}}(A_{k_1}^{\circ}) \cap S_{\mathbf{e_2}}(A_{k_2}^{\circ}) = \emptyset
\quad \text{whenever } \mathbf{e_1} \neq \mathbf{e_2}.
\]
Hence,
\[
S_{\mathbf{e_1}}(J_{k_1}) \cap S_{\mathbf{e_2}}(A_{k_2}) = \emptyset.
\]

    
    Since $S_{\mathbf{e_1}}(x_{k_1})\in S_{\mathbf{e_1}}(\mathcal{O}_{k_1})$, $S_{\mathbf{e_1}}(x_{k_1})\in S_{\mathbf{e_2}}(A_{k_2})$, $S_{\mathbf{e}_2}(A_{k_2}) \cap S_{\mathbf{e_1}}(J_{k_1})=\emptyset$ and using the fact $A_{k_2}$ is connected, we get $S_{\mathbf{e_2}}(A_{k_2})\subseteq  S_{\mathbf{e_1}}(\mathcal{O}_{k_1}).$ Now, we get $$S_{\mathbf{e_2}}(x_{k_2}) \in S_{\mathbf{e_2}}(\mathcal{O}_{k_2}) \backslash (S_{\mathbf{e_1}}(A_{k_1})\cup S_{\mathbf{e_2}}(A_{k_2}) )  \subset B[z,\epsilon].$$ Let $ \#(\cup_{k=1}^n E_{i,k}^N)=l.$ Consequently, we can choose all distinct $\mathbf{e_1},\mathbf{e_2},\dots,\mathbf{e}_l \in \cup_{k=1}^n E_{i,k}^N$ and $k_{1},k_{2},\dots,k_{l}\in V$ such that 
    $$ S_{\mathbf{e_M}}(x_{k_M}) \in S_{\mathbf{e_{M}}}(\mathcal{O}_{k_M})\backslash (\cup_{1 \le i \le M} S_{\mathbf{e}_i}(A_{k_i})) \subset B[z.\epsilon]$$
    for any $1 \le M \le l.$ Thus, we get 
    $$S_{\mathbf{e}_l}(x_{k_l}) \in S_{\mathbf{e}_l} (\mathcal{O}_{k_l}) \backslash (\cup_{1 \le i \le l} S_{\mathbf{e}_i}(A_{k_i})) = S_{\mathbf{e}_l} (\mathcal{O}_{k_l}) \backslash A_i  \subset B[z,\epsilon].$$ This implies that $B[z,\epsilon]$ is not fully contained in $A_i$, which is a contradiction.\\\\
    (ii) Let $\partial A_i=V_1 \cup V_2$ be a separation of $\partial A_i,$ $V_1$ and $V_2$ are closed sets. By Lemma \ref{GTL}, for any $v_i  \in V_i (i=1,2)$  and any $0<\epsilon < \inf \{|x-y|: x\in V_1, y \in V_2\}$ we can choose a simple closed curve $\gamma_i$ which does not intersect $\partial A_i$ and which separates $v_1$ and $v_2.$ From the part (i), $A_i^0$ has no hole. We can claim that $\gamma_i \cap A_i^0=\emptyset$. Otherwise, we get some $x \in \gamma_i \cap A_i^0$.\\ Since $\gamma_i \cap \partial A_i=\emptyset, \gamma_i \subset A_i^0.$ We may consider that $v_1 \subseteq \gamma_i^0,$ then for $\delta>0$ the ball $V_1(v_1,\delta) \nsubseteq A_i^0$. This is contradictory to the fact that $A_i^0$ has no holes. So necessarily, we get $\gamma_i \cap A_i=\gamma_i \cap A_i^0=\emptyset.$ Let $A_{i,1}$ and $A_{i,2}$ consist of the points of $A_i$ which are contained in the interior and exterior of $\gamma_i$, respectively. Then $A_i=A_{i,1} \cup A_{i,2}$ is a separation, which contradicts the connectedness of $A_i.$ 
\end{proof}
   \begin{theorem}\label{TH2.2}
    \textit{Let} $\mathcal{S}=(G(V,E),(S_{e})_{e\in E})$
\textit{is a GDIFS in} $\mathbb{R}^{d}$ \textit{such that }$%
S_{e}$\textit{\ is one-to-one for all }$e\in E$\textit{\ and }$%
(A_{1},...,A_{n})\in (P_{cp}(\mathbb{R}^{d}))^{n}$\textit{\ its attractor.}
     Assume that the set $A_i \backslash \Theta (A_i) \neq \emptyset$ for all $i\in V,$ where $\Theta (A_i)$ is the collection of trivial connected components of $A_i$ for all $i\in V.$ Then $A_i \backslash \Theta (A_i)$ is dense in $A_i$ for all $i\in V.$ This implies that either $\overline{\dim}_B (A_i \backslash \Theta (A_i))=\overline{\dim}_B (A_i)$ and $\underline{\dim}_B (A_i \backslash \Theta (A_i))=\underline{\dim}_B (A_i)$ or $A_i \backslash \Theta (A_i) = \emptyset$ for all $i\in V.$
\end{theorem}
\begin{proof}
Given that $A_i \backslash \Theta (A_i) \neq \emptyset$ for all $i\in V$. If $x_i \in A_i \backslash \Theta (A_i)$, then there is a unique non-trivial connected componet $\Upsilon(x_i)$ of $A_{i}$ containing $x_{i}.$ Since the map $S_{e}$ is injective and contraction for each $e\in E_{i,j}$, we have $S_e(\Upsilon(x_i))$ is a non-tivial connected component of $A_i$. Thus, $S_e(A_j \backslash \Theta (A_j))\subseteq A_i \backslash \Theta (A_i)$ for all $e\in E_{i,j}.$ Let $x\in A_i.$ First, we will show that there is an infinite sequence $\mathbf{e}=(e_1,e_2,\dots,e_m,\dots)\in E_{i}^*$ such that $$ \lim_{ m \to \infty} S_{e_1} \circ S_{e_2}\circ \cdots  \circ S_{e_m}(A_{k_m})=\bigcap_{m=1}^{\infty} S_{e_1} \circ S_{e_2}\circ \cdots  \circ S_{e_m}(A_{k_m})= \{x\}.$$ Since $x\in A_{i}$ and  $A_i=\bigcup_{k=1}^n \bigcup_{e \in E_{i,k}} S_e (A_k)$, there exists $k_{1}\in V$ and $e_{1}\in E_{i,k_1}$ such that $x\in S_{e_1}(A_{k_1}).$ Thus, there exists $x_1\in A_{k_1}$ such that $x=S_{e_1}(x_1).$ Again, since $x_1\in A_{k_1}$, there exists $k_{2}\in V$ and $e_{2}\in E_{k_1,k_2}$ such that $x_{1}\in S_{e_2}(A_{k_2}).$ Thus, there exists $x_2\in A_{k_2}$ such that $x_{1}=S_{e_2}(x_2).$ This implies that $x=S_{e_1}\circ S_{e_{2}}(x_2)$. Thus, we get $x\in S_{e_1}\circ S_{e_{2}} (A_{k_2})$ and $S_{e_1}\circ S_{e_{2}} (A_{k_2})\subset S_{e_1}(A_{k_1})\subset A_i$. By using same procedure, we get an infinite sequence $\mathbf{e}=(e_1,e_2,\dots,e_m,\dots)\in E_{i}^*$ such that $x\in S_{e_1} \circ S_{e_2}\circ \cdots  \circ S_{e_m}(A_{k_m})$ for all $m\in \mathbb{N}$ and $S_{e_1} \circ S_{e_2}\circ \cdots  \circ S_{e_{m+1}}(A_{k_{m+1}})\subset S_{e_1} \circ S_{e_2}\circ \cdots  \circ S_{e_m}(A_{k_m})\subset A_i$ for all $m\in \mathbb{N}.$ Since $x\in S_{e_1} \circ S_{e_2}\circ \cdots  \circ S_{e_m}(A_{k_m})$ for all $m\in \mathbb{N}$ and $\{S_{e_1} \circ S_{e_2}\circ \cdots  \circ S_{e_m}(A_{k_m})\}_{m=1}^{\infty}$ is a monotonic decreasing sequence of non-empty compact sets with $\text{diam}(S_{e_1} \circ S_{e_2}\circ \cdots  \circ S_{e_m}(A_{k_m}))\to 0$ as $m\to \infty$, by Cantor's intersection theorem we get $$ \lim_{ m \to \infty} S_{e_1} \circ S_{e_2}\circ \cdots  \circ S_{e_m}(A_{k_m})=\bigcap_{m=1}^{\infty} S_{e_1} \circ S_{e_2}\circ \cdots  \circ S_{e_m}(A_{k_m})= \{x\}.$$ Next, we will show that $\lim_{ m \to \infty} S_{e_1} \circ S_{e_2}\circ \cdots  \circ S_{e_m}(a_{k_m})=x,$ where $a_{k_m}$ is an arbitrary element of $A_{k_m}$. Set $r: =\max\{r_e: e\in E\}$. For the given $\epsilon >0$, choose $N\in \mathbb{N}$ such that $r^{N}\text{diam}(\mathbf{A})\leq \epsilon.$ For $m\geq N$, we have 
\begin{align*}
  \|S_{e_1} \circ S_{e_2}\circ \cdots  \circ S_{e_m}(a_{k_m})-x\|&\leq r^{m}\|a_{k_m}-\lim_{n \to \infty} S_{e_{m+1}} \circ S_{e_{m+2}}\circ \cdots  \circ S_{e_{m+n}}(A_{k_{m+n}})\|\\&\leq r^{m} \text{diam}(\mathbf{A})\leq \epsilon. \end{align*}
Thus, we get $\lim_{ m \to \infty} S_{e_1} \circ S_{e_2}\circ \cdots  \circ S_{e_m}(a_{k_m})=x,$ where $a_{k_m}\in A_{k_m}.$  If $A_i \backslash \Theta (A_i) \neq \emptyset$ for all $i\in V$, then we consider a sequence $\{S_{e_1} \circ S_{e_2}\circ \cdots  \circ S_{e_m}(x_{k_m})\}_{m=1}^{\infty},$ where $x_{k_m}\in A_{k_m}\backslash \Theta (A_{k_m}).$ Since $S_e(A_j \backslash \Theta (A_j))\subseteq A_i \backslash \Theta (A_i)$ for all $e\in E_{i,j},$ we have $S_{e_1} \circ S_{e_2}\circ \cdots  \circ S_{e_m}(x_{k_m})\in A_i \backslash \Theta (A_i)$ for all $m\in \mathbb{N}$. Since $\lim_{ m \to \infty} S_{e_1} \circ S_{e_2}\circ \cdots  \circ S_{e_m}(a_{k_m})=x,$ we have $\overline{A_i \backslash \Theta (A_i)}=A_i.$ This implies that $A_i \backslash \Theta (A_i)$ is dense in $A_i$ for all $i\in V.$
It is well known (See \cite{F}) that the upper and lower box dimensions are invariant under closure. Thus, we obtain  $\overline{\dim}_B (A_i \backslash \Theta (A_i))=\overline{\dim}_B (A_i)$ and $\underline{\dim}_B (A_i \backslash \Theta (A_i)=\underline{\dim}_B (A_i).$ This completes the proof.
\end{proof}
Now, we discuss some results of GDIFS associated with dimension-preserving bi-Lipschitz maps. The bi-Lipschitz maps offer a good balance, which leads to categories that are interesting and intriguing both geometrically and algebraically. In fractal geometry, the dimension is a key idea that is used to distinguish between fractal sets. If two sets have different dimensions (such as Hausdorff dimension, box dimension, or others), we usually see them as different. However, even if two compact sets share the same dimension, they can still be quite different in many aspects. Therefore, it is natural to look for another property that can tell us when two fractal sets are truly similar. One such property often considered for this purpose is bi-Lipschitz equivalence.

\begin{theorem}
   Let $(A_1,A_2,\dots,A_n)\in (P_{cp}(\mathbb{R}^d))^n$ be the attractor of the GDIFS
$\mathcal S=(G(V,E),(S_e)_{e\in E})$, where $S_e$ is a bi-Lipschitz map and $\mathcal S$ satisfies the OSC. Suppose that $\dim_H(A_i)>0$ for all $i \in V.$  Then $$\dim_H(A_i \cap U_i)=\dim_H(A_i) \iff A_i \cap U_i \neq \emptyset~ \text{for all}~ i \in V,$$
   where $(U_1,U_2,\dots,U_n)$ is the $n$-tuple of open sets satisfying the open set condition. 
\end{theorem}
\begin{proof}
    On the one hand, since $\dim_H(A_i)>0$, we infer that $\dim_H(A_i \cap U_i)>0$, so $A_i \cap U_i \neq \emptyset$. On the other hand, let $A_i \cap U_i \neq \emptyset$ and $x_0 \in U_i \cap A_i.$ Since, $U_i$ is an open set, there exists a $\delta>0$ such that $B(x_0, \delta) \subset U_i.$ One can see that $A_i=\bigcup_{k=1}^{n}\bigcup_{\mathbf{e}\in E_{i,k}^m}S_{\mathbf{e}}(A_k)$ for each $m\in \mathbb{N}.$  Set $r:=\max \{r_{e}: e \in E \}$, where $r_e \in (0,1)$ is the upper bi-Lipschitz constant of the map $S_e$. Clearly $r\in (0,1)$. Then, there exists $N>0$ such that $r^N \text{diam} (\mathbf{A})< \frac{\delta}{2}.$ This implies $r^N \text{diam}(A_i) \le r^N \text{diam} (\mathbf{A}) < \frac{\delta}{2}$ for all $i\in V.$ Thus, $\text{diam} (S_{\mathbf{e}}(A_k)) < \frac{\delta}{2}$ for all $\mathbf{e} \in E_{i,k}^N.$ Since $x_0\in A_i$ and  $A_i=\bigcup_{k=1}^{n}\bigcup_{\mathbf{e}\in E_{i,k}^N}S_{\mathbf{e}}(A_k)$, there exist a $k_1$ and $\mathbf{e}\in E_{i,k_1}^{N}$ such that $x_0\in S_{\mathbf{e}}(A_{k_1}).$ Clearly, $S_{\mathbf{e}}(A_{k_1})\subseteq B(x_0, \delta)$ and $S_{\mathbf{e}}(A_{k_1})\subseteq A_i$. Thus, we have
    \begin{equation}\label{EQ2}
        S_{\mathbf{e}}(A_{k_1}) \subset B(x_0, \delta) \cap A_i \subset U_i \cap A_i \subset A_i.
    \end{equation}
Given that the map $S_e$ is a bi-Lipschitz map, thus by using the Lemma \ref{FLemma},  we get $\dim_H(S_{\mathbf{e}}(A_{k_1}))=\dim_H(A_{k_1}).$ From the above expression, we have $\dim_H(A_{k_{1}})=\dim_H(S_{\mathbf{e}}(A_{k_1})) \le \dim_H(U_i \cap A_i) \le \dim_H(A_i).$ The GD-IFS is strongly connected, thus by Lemma \ref{equaldiem},  we obtain $\dim_{H}(A_{k_{1}})=\dim_H(A_i)$. This implies that $\dim_H(A_i)=\dim_H(U_i \cap A_i).$ 
    Hence, the proof is completed.
\end{proof}
  Here, first, we discuss the Lipschitz embedding of two graph-directed attractors under some separation condition, and later, we discuss the bi-Lipschitz equivalence under certain conditions. 
\begin{theorem} \label{THML}Let $(A_1,A_2,\dots,A_n)\in (P_{cp}(\mathbb{R}^d))^n$ be the attractor of the GDIFS
$\mathcal S=(G(V,E),(S_e)_{e\in E})$, where $S_e$ is a bi-Lipschitz map such that
$$l_{e}\|x-y\|\leq \|S_e(x)-S_{e}(y)\|\leq r_{e}\|x-y\|,$$ where $l_e,r_e\in (0,1),~ \forall~e\in E.$ 
Let $(B_1,B_2,\dots,B_n)\in (P_{cp}(\mathbb{R}^d))^n$ be the attractor of another GDIFS~ $\mathcal T=(G(V,E),(f_e)_{e\in E})$, where $f_e$ is a Lipschitz map such that 
$$\|f_e(x)-f_{e}(y)\|\leq l'_{e}\|x-y\|,$$ where $l'_{e}\in (0,1)$ and $l'_{e}\leq l_{e}$ for all $e\in E.$
If both graph-directed IFSs satisfy the SSC, then for each $i\in V$ there exists a Lipschitz bijection from $A_i$ to $B_i$.
\end{theorem}
\begin{proof} Let $x\in A_i$. Since the graph-directed IFS  $\mathcal S$ satisfies the SSC and by using same idea as in Thoerem \ref{TH2.2}, we get an unique infinite sequence $\mathbf{e}=(e_1,e_2,\dots,e_m,\dots)\in E_{i}^*$ such that $$ \lim_{ m \to \infty} S_{e_1} \circ S_{e_2}\circ \cdots  \circ S_{e_m}(A_{k_m})=\bigcap_{m=1}^{\infty} S_{e_1} \circ S_{e_2}\circ \cdots  \circ S_{e_m}(A_{k_m})= \{x\}.$$ 
Now, we define a bijective map $\Phi: A_i\to B_i $ such that $$\Phi(x)=\bigcap_{m=1}^{\infty} f_{e_1} \circ f_{e_2}\circ \cdots  \circ f_{e_m}(B_{k_m}),$$
where $\mathbf{e}=(e_1,e_2,\dots,e_m,\dots)\in E_{i}^*$ is a unique\cite{GC} infinite sequence such that  $\bigcap_{m=1}^{\infty} S_{e_1} \circ S_{e_2}\circ \cdots  \circ S_{e_m}(A_{k_m})= \{x\}$. Thus, the map $\Phi$ is well defined. The bijectiveness of the map $\Phi$ follows from Lemma 1.3.5 \cite{GC}. Now, we will show that the map $\Phi$ is a Lipschitz map. Let $x$ and $x'$ be two distinct points of $A_i$. Since the IFS $\mathcal S$ satisfies the SSC, there exist $m\in \mathbb{N}$ and $e\in E_{k_m,j}, e'\in E_{k_m,j'}$ such that $e\ne e'$ and $$x\in S_{e_1} \circ S_{e_2}\circ \cdots  \circ S_{e_m}\circ S_{e}(A_{j}),~~x'\in S_{e_1} \circ S_{e_2}\circ \cdots  \circ S_{e_m}\circ S_{e'}(A_{j'}).$$ 
Thus, we get $z\in S_{e}(A_{j})$ and $z'\in S_{e'}(A_{j'})$ such that $z\ne z',$ $x=S_{e_1} \circ S_{e_2}\circ \cdots  \circ S_{e_m}(z)$ and $ x'=S_{e_1} \circ S_{e_2}\circ \cdots  \circ S_{e_m}(z').$ Therefore, we have 
$$l_{e_1}l_{e_2}\cdots l_{e_m}\|z-z'\|\leq \|x-x'\|\leq r_{e_1}r_{e_2}\cdots r_{e_m}\|z-z'\|.$$ As  the IFS $\mathcal S$ satisfies SSC, we have $$D_1:=\min\limits_{\substack{e\in E_{i,j},e'\in E_{k,l},\\ e\ne e', \forall i,j,k,l\in V}}\{ d(S_{e}(A_j), S_{e'}(A_l))\}>0.$$ Therefore from the above, we obtain
$$l_{e_1}l_{e_2}\cdots l_{e_m} D_1\leq \|x-x'\|\leq r_{e_1}r_{e_2}\cdots r_{e_m} \max_{i\in V}\{\text{diam}(A_i)\}.$$
By similar idea, we get
$$\|\Phi(x)-\Phi(x')\|\leq l'_{e_1}l'_{e_2}\cdots l'_{e_m} \max_{i\in V}\{\text{diam}(B_i)\}.$$ Since $l'_{e}\leq l_{e}$ for all $e\in E$ and using previous two inequalites, we get
$$ \|\Phi(x)-\Phi(x')\|\leq D_{1}^{-1} \max_{i\in V}\{\text{diam}(B_i)\}\|x-x'\|.$$
This implies the map $\Phi$ is a Lipschitz map.
This completes the proof.
\end{proof}
In the following theorem, we discuss the bi-Lipschitz euvalence of two different graph-directed attractors on the same graph under BDP. 
\begin{theorem}
Let $(A_1,A_2,\ldots,A_n)\in (P_{cp}(\mathbb{R}^d))^n$ be the attractor of the GDIFS
\[
\mathcal S=(G(V,E),(S_e)_{e\in E}),
\]
where $S_e$ is a bi-Lipschitz map for each $e\in E$. Let
$(B_1,B_2,\ldots,B_n)\in (P_{cp}(\mathbb{R}^d))^n$ be the attractor of another GDIFS
\[
\mathcal T=(G(V,E),(f_e)_{e\in E}),
\]
where $f_e$ is a bi-Lipschitz map for each $e\in E$. Suppose that both graph-directed IFSs satisfy the SSC and that the GDIFS $\mathcal S$ satisfies the BDP.

Assume further that, for every finite admissible path $w$, the corresponding composed maps $S_w$ and $f_w$ satisfy
\[
\operatorname{Lip}^{-}(S_w)=\operatorname{Lip}^{-}(f_w)
\]
and
\[
\operatorname{Lip}^{+}(S_w)=\operatorname{Lip}^{+}(f_w).
\]
Then $A_i$ and $B_i$ are bi-Lipschitz equivalent for each $i\in V$.
\end{theorem}
\begin{proof}
First, we define the same bijective map $\Phi$ as in the previous result. The map $\Phi: A_i\to B_i $ is as follows $$\Phi(x)=\bigcap_{m=1}^{\infty} f_{e_1} \circ f_{e_2}\circ \cdots  \circ f_{e_m}(B_{k_m}),$$
where $\mathbf{e}=(e_1,e_2,\dots,e_m,\dots)\in E_{i}^*$ is a unique infinite sequence such that  $\bigcap_{m=1}^{\infty} S_{e_1} \circ S_{e_2}\circ \cdots  \circ S_{e_m}(A_{k_m})= x$. This map is well defined. After that, using the same line proof as in the previous theorem, we obtain
$$\operatorname{Lip}^{-}(S_w) D_1\leq \|x-x'\|\leq \operatorname{Lip}^{+}(S_w) \max_{i\in V}\{\text{diam}(A_i)\},$$ where $w=e_1e_2\cdots e_m.$
As the IFS $\mathcal{T}$ satisfies SSC, we have $$D_2:=\min\limits_{\substack{e\in E_{i,j},e'\in E_{k,l},\\ e\ne e', \forall i,j,k,l\in V}}\{d(f_{e}(B_j), f_{e'}(B_l))\}>0.$$
By similar idea, we get
$$\operatorname{Lip}^{-}(f_w) D_2\leq \|\Phi(x)-\Phi(x')\|\leq \operatorname{Lip}^{+}(f_w) \max_{i\in V}\{\text{diam}(B_i)\}.$$
Since
\[
\operatorname{Lip}^{-}(S_w)=\operatorname{Lip}^{-}(f_w)
\]
and
\[
\operatorname{Lip}^{+}(S_w)=\operatorname{Lip}^{+}(f_w),
\]
and $\mathcal{S}$ satisfies BDP, we get 

$$\operatorname{Lip}^{+}(S_w) \le K \operatorname{Lip}^{-}(S_w),$$ which implies
$$\operatorname{Lip}^{+}(f_w) \le K \operatorname{Lip}^{-}(f_w).$$
Hence, we have
$$K^{-1} D_{2} ( \max_{i\in V}\{\text{diam}(A_i)\})^{-1}\|x-x'\|\leq \|\Phi(x)-\Phi(x')\|\leq K D_1^{-1} \max_{i\in V}\{\text{diam}(B_i)\}\|x-x'\|,$$ where $K>0$ is the BDP constant. Thus, $\Phi$ is bi-Lipschitz. This completes the proof.
\end{proof}

\begin{remark}
We note that the bounded distortion property (BDP) is a genuine additional assumption in the bi-Lipschitz setting and is not automatically satisfied by a family of bi-Lipschitz contractions.

Indeed, consider the map
\[
S:\mathbb{R}^{2}\to\mathbb{R}^{2},
\qquad
S(x,y)=\left(\frac{x}{2},\frac{y}{3}\right).
\]
Then \(S\) is a contractive bi-Lipschitz map satisfying
\[
\frac{1}{3}\|(x,y)-(x',y')\|
\le
\|S(x,y)-S(x',y')\|
\le
\frac{1}{2}\|(x,y)-(x',y')\|.
\]
However, for the \(n\)-th iterate,
\[
S^{n}(x,y)
=
\left(
\frac{x}{2^{n}},
\frac{y}{3^{n}}
\right),
\]
we have
\[
\operatorname{Lip}^{+}(S^{n})=\frac{1}{2^{n}},
\qquad
\operatorname{Lip}^{-}(S^{n})=\frac{1}{3^{n}}.
\]
Consequently,
\[
\frac{\operatorname{Lip}^{+}(S^{n})}
{\operatorname{Lip}^{-}(S^{n})}
=
\left(\frac{3}{2}\right)^{n}
\to\infty
\qquad
\text{as } n\to\infty.
\]
Hence, no uniform distortion constant exists and the bounded distortion property fails. Thus, bi-Lipschitzity alone does not imply the bounded distortion property.

On the other hand, Examples 6.5 and 6.6 in \cite{NG} provide classes of contractive bi-Lipschitz systems satisfying the BDP. Therefore, the bounded distortion property is a natural but nontrivial assumption in the bi-Lipschitz setting.
\end{remark}
\begin{theorem}
 Let $(A_1,A_2,\dots,A_n)\in (P_{cp}(\mathbb{R}^d))^n$ be the attractor of the GDIFS
$\mathcal S=(G(V,E),(S_e)_{e\in E})$, where $S_e$ is a bi-Lipschitz map such that
$$l_{e}\|x-y\|\leq \|S_e(x)-S_{e}(y)\|\leq r_{e}\|x-y\|,$$ where $l_e,r_e\in (0,1)~\text{for all}~e\in E.$ 
    Let $\mathcal{S}$ satisfies the SSC and bounded distortion property (BDP) , then $A_i$ and $A_j$ are bi-Lipschitz equivalent for each $i,j\in V.$
\end{theorem}
\begin{proof}
Suppose that $ A \in \{A_i\}_{i \in V}$. Then there exist bi-Lipschitz maps $L_0, L_1$ and a compact set $C$, such that
\begin{equation} \label{Eq1}
    A=L_0(A) \cup L_1(A) \cup C,
\end{equation}
 where the union is disjoint. Since the underlying directed graph is strongly connected, for every vertex \(j\in V\), there exists an admissible finite path from the vertex corresponding to \(A\) to \(j\). Hence, finite compositions of the corresponding bi-Lipschitz contractions generate cylinder sets associated with admissible finite paths. Moreover, under the strong separation condition, these cylinder sets are pairwise disjoint whenever they correspond to distinct admissible paths. Therefore, there exist non-empty families of bi-Lipschitz maps \(\{\Gamma_j\}_{j\in V}\) such that
 \begin{equation}\label{Eq2}
     A=\bigcup_{j \in V} \bigcup_{L \in \Gamma_j } L (A_j),
 \end{equation}
where the union is disjoint.\\
Now, let  $\{S_i\}_{i=0}^m$ be the collection of bi-Lipschitz maps such that $\{S_i(A)\}_{i=0}^m$ are pairwise disjoint. Then we show that $A$ and $ \cup_{i=0}^k S_i(A)$ are bi-Lipschitz  equivalent.
We will prove this by induction. It suffices to show the conclusion for $m=1,$ that is $A$ is bi-Lipschitz equivalen to $S_0(A) \cup S_1(A).$ For a finite word $i_1...i_m \in \{0,1\}^m$, set $L_{i_1...i_m}=L_{i_1} \circ \cdots \circ L_{i_m}.$ For the empty word $w$, $L_w$ equals the identity mappings. Also, $1^m$ is an abbreviation of $1...1 (m~ \text{ones})$. From Equation \ref{Eq1}, we have $$A=(L_0(A) \cup C) \cup L_1(A).$$ Now, apply $A$ on the map $L_1$  and continue the process, we get
$$A=\cup_{m=0}^\infty \left((L_{1^m0}(A) \cup L_{1^m} (C) \right) \cup \{t\},$$ where $t$ is the fixed point of the bi-Lipschitz map $L_1.$ In this way, we can write
$$A=L_0(A) \cup_{m=0}^\infty L_{1^{m+1}0}(A) \cup \left( \cup_{m=0}^\infty L_{1^m}(C) \cup \{t\}\right)=:L_0 (A) \cup A' \cup C',$$ where $A' \cup C'=L_1 (A) \cup C,$ and
$$S_0(A)\cup S_1 (A)=S_0(A) \cup_{m=0}^\infty S_1L_{1^{m}0}(A) \cup \left( \cup_{m=0}^\infty S_1L_{1^m}(C) \cup \{S_1t\}\right)=:S_0 (A) \cup A''\cup C'',$$ where $A'' \cup C''=S_1(A).$\\
Now, we define a bijection map $\psi: A \to S_0(A) \cup S_1(A)$ by 
\begin{equation*}
\psi(x)=\begin{cases}
S_0L_0^{-1}(x)  \quad \text{if}~x \in L_0(A),\\
 S_1L_1^{-1}(x)  \quad \text{if}~x \in A'=\cup_{m=0}^\infty L_{1^{m+1}0}A,\\
S_1(x)  \quad \text{if}~x \in C'=\cup_{m=0}^\infty L_{1^m}(C) \cup \{t\}.\\
\end{cases}
\end{equation*}
Now, we will show that the map $\psi$ is bi-Lipschitz. Since $L_1$ and $S_1$ are bi-Lipschitz maps such that  $$l_1\|x-y\|\leq \|L_1(x)-L_1(y)\|\leq r_1\|x-y\|$$ and $$l'_1\|x-y\|\leq \|S_1(x)-S_1(y)\|\leq r'_1\|x-y\|,$$ where $l_1, r_1,l'_1,r_1' \in (0,1).$
The map $L_0$ is also bi-Lipschitz so we have $$l_0\|x-y\|\leq \|L_0(x)-L_0(y)\|\leq r_0\|x-y\|,$$ where $l_0, r_0 \in (0,1).$
Since, $d(L_0 (A), A'\cup C')>0$ and $d(S_0 (A), A'' \cup C'')>0,$ $d$ is Huasdorff distance, so here we only need the restriction of $\psi$ to $A' \cup C'$ (the corresponding image is $A'' \cup C''$). Put $D=:\min \{d(L_0(A), L_1(A)), d(L_0(A), C), d(L_1(A), C)\}>0.$ For $x \in A'$ and $x' \in C'$, let $x \in L_{1^{p+1}0} (A)=L_{1^{p+1}}(L_0 (A))$ with $p \ge 0$ and let $x' \in L_{1^m} (C)$
 with $m \ge 0$ or $m=\infty$ and $L_{1^\infty} (C)=\{t\}.$  Then $\psi (x) \in S_1 L_{1^{p}0} (A)$ and $\psi (x') \in S_1 L_{1^m} (C).$
 Therefore, we have 
 $$ l_1^{\min \{p+1,m\}} D\le \|x-x'\| \le r_1^{\min \{p+1,m\}} \text{diam} (A).$$
 By similar idea, we get $$l_1'l_1^{\min \{p,m\}+1} D\le \|\psi(x)-\psi(x')|\  \le r_1' r_1^{\min \{p,m\}+1} \text{diam} (A).$$
 Consequently, for any $x \in A'$ and $x' \in C'$, by using the above two inequlaities and BDP 
 $$  (K')^{-1}l_1' D (\text{diam}(A))^{-1} \|x-x'\| \le\|\psi(x)-\psi(x')\| \le r_1' K' D^{-1} \text{diam}(A) \|x-x'\|,$$ where $K'>0$ is the BDP constant. Thus $\psi$ is bi-Lipschitz. This completes our claim.
 Now, let $\{T_j\}_{ j \in V}$ be collection of bi-Lipschitz map such that $\{T_j (A_j)\}_{j \in V}$ are pairwise disjoint. So from above, we say that $ \bigcup_{L \in \Gamma_j } L (A_j)$ and $T_j (A_j)$ are bi-Lipschitz equivalent. Then for any $A \in \{A_i\}_{i \in V}$ and using Eqaution \ref{Eq2}, we have that $A=\bigcup_{j \in V} \bigcup_{L \in \Gamma_j } L (A_j)$  and $\cup_{j \in V}T_j (A_j)$ are bi-Lipschitz equivalent. That is $A_i$ and $A_j$ are bi-Lipschitz equivalent for each $i,j \in V.$ This completes the proof.
 \end{proof}

	\bibliographystyle{amsplain}


 \end{document}